\documentclass[12pt,leqno]{article}
\usepackage{tikz}
\usepackage{hyperref}
\usepackage{empheq}
\usepackage[titletoc]{appendix}
\usepackage{enumerate}
\usepackage[doc]{optional}
\usepackage{xcolor}
\usepackage{float}
\usepackage{soul}
\usepackage{graphicx}
\definecolor{labelkey}{rgb}{0,0.08,0.45}
\definecolor{refkey}{rgb}{0,0.6,0.0}
\definecolor{Brown}{rgb}{0.45,0.0,0.05}
\definecolor{lime}{rgb}{0.00,0.8,0.0}
\definecolor{lblue}{rgb}{0.5,0.5,0.99}

 \usepackage{mathpazo}
\usepackage{nameref}
\usepackage{amsthm}
\usepackage{amsmath}

\usepackage{stmaryrd}
\usepackage{amssymb}

\newcommand{\aref}[1]{\hyperref[#1]{Appendix~\ref{#1}}}

\newcommand{\nnn}{\ensuremath{{n\in{\mathbb N}}}}

\newcommand{\menge}[2]{\big\{{#1}~\big |~{#2}\big\}}

\newcommand{\fenv}[1]%
{\ensuremath{\,\overrightarrow{\operatorname{env}}_{#1}}}
\newcommand{\benv}[1]%
{\ensuremath{\,\overleftarrow{\operatorname{env}}_{#1}}}

\newcommand{\RR}{\ensuremath{\mathbb R}}

\newcommand{\NN}{\ensuremath{\mathbb N}}

\providecommand{\BB}[2]{\operatorname{ball}(#1;#2)}

\newcommand{\ran}{\ensuremath{\operatorname{ran}}}
\newcommand{\zer}{\ensuremath{\operatorname{zer}}}

\newcommand{\Id}{\ensuremath{\operatorname{Id}}}

\newcommand{\TAB}{T_{(A,B)}}

{\begin{list}{}{%
\settowidth{\labelwidth}{\textrm{#1~}}%
\setlength{\leftmargin}{\labelwidth+\labelsep}}}%requires macro calc.sty
{\end{list}}

\usepackage{amsthm}
\usepackage[capitalize]{cleveref}
\crefname{equation}{}{equations}
\crefname{chapter}{Appendix}{chapters}
\crefname{item}{}{items}

\newtheorem{theorem}{Theorem}[section]
\newtheorem{lemma}[theorem]{Lemma}
\newtheorem{lem}[theorem]{Lemma}

\newtheorem{prop}[theorem]{Proposition}

\newtheorem{thm}[theorem]{Theorem}%[section]
\newtheorem{example}[theorem]{Example}
\newtheorem{ex}[theorem]{Example}
\newtheorem{fact}[theorem]{Fact}
\def\endproof{\ensuremath{\hfill \quad \blacksquare}}

\providecommand{\ds}{\displaystyle}

\providecommand{\abs}[1]{\lvert#1\rvert}
\providecommand{\Abs}[1]{\Big\lvert#1\Big\rvert}
\providecommand{\norm}[1]{\lVert#1\rVert}
\providecommand{\normsq}[1]{\lVert#1\rVert^2}
\providecommand{\bk}[1]{\left(#1\right)}
\providecommand{\stb}[1]{\left\{#1\right\}}
\providecommand{\innp}[1]{\langle#1\rangle}

\providecommand{\RA}{\Rightarrow}

\providecommand{\RR}{\mathbb{R}}

\providecommand{\ran}{\operatorname{ran}}

\providecommand{\intr}{\operatorname{int}}

\providecommand{\gap}{\operatorname{v}}

\newcommand{\fix}{\ensuremath{\operatorname{Fix}}}

\providecommand{\parl}{\operatorname{par}}

\providecommand{\Id}{\operatorname{{ Id}}}

\providecommand{\fady}{\varnothing}

\providecommand{\rras}{\rightrightarrows}

\providecommand{\NN}{\mathbb{N}}
\providecommand{\BB}[2]{\operatorname{ball}(#1;#2)}

\providecommand{\fix}{\operatorname{Fix}}
\providecommand{\ran}{\operatorname{ran}}
\providecommand{\rec}{\operatorname{rec}}
\providecommand{\Id}{\operatorname{Id}}

\providecommand{\pt}{{\partial}}

\providecommand{\zer}{\operatorname{zer}}

\providecommand{\T}{{ T}}

\providecommand{\fady}{\varnothing}

\providecommand{\RR}{\mathbb{R}}
\providecommand{\NN}{\mathbb{N}}

\providecommand{\linop}{L}
\newcommand{\DRS}[2]{T_{{#1},{#2}}}

\makeatletter
\def\namedlabel#1#2{\begingroup
   \def\@currentlabel{#2}%
   \label{#1}\endgroup
}
\makeatother

\definecolor{myblue}{rgb}{.8, .8, 1}
  \newcommand*\mybluebox[1]{%
    \colorbox{myblue}{\hspace{1em}#1\hspace{1em}}}

\allowdisplaybreaks % or locally if problems {\allowdisplaybreaks
\newenvironment{myproof}[1][\proofname]{
{\emph {#1} }%
}{\endproof}

\begin{document}
%-------------------------------------------------------------------------

\title{\textsc
The Douglas-Rachford algorithm for two (not necessarily
intersecting) affine subspaces}

\author{
Heinz H.\ Bauschke\thanks{
Mathematics, University
of British Columbia,
Kelowna, B.C.\ V1V~1V7, Canada. E-mail:
\texttt{heinz.bauschke@ubc.ca}.}
~and Walaa M.\ Moursi\thanks{
Mathematics, University of
British Columbia,
Kelowna, B.C.\ V1V~1V7, Canada. E-mail:
\texttt{walaa.moursi@ubc.ca}.}}

\date{April 14, 2015}
\maketitle
\begin{abstract}
\noindent
The Douglas--Rachford algorithm is a classical and
very successful 
splitting method for 
finding the zeros of the sums of  
monotone operators. When the underlying operators are normal cone
operators,
the algorithm solves a convex feasibility problem.
In this paper, we provide a detailed study of the Douglas--Rachford 
iterates and the corresponding {shadow sequence} when the 
sets are affine subspaces that do not necessarily intersect.
We prove strong convergence of the shadows to the nearest
generalized solution.
Our results extend recent work from the consistent to the
inconsistent case. 
Various examples are provided to illustrates the results. 
\end{abstract}
{\small

}

\section{Introduction}

Throughout this paper
\begin{empheq}[box=\mybluebox]{equation*}
\label{T:assmp}
X \text{~~is a real Hilbert space},
\end{empheq} 
with inner product $\innp{\cdot,\cdot}$ and
induced norm $\norm{\cdot}$. A (possibly) set-valued operator 
$A:X\rras X$ is \emph{monotone} if any two pairs 
$(x,u)$ and $(y,v)$ in the graph of $A$
satisfy
$\innp{x-y,u-v}\ge 0$, and 
 is \emph{maximally monotone} 
if it is monotone and 
any proper enlargement of the graph of 
$A$ (in terms of set inclusion) 
destroys the monotonicity of $A$.
Monotone operators play an important role in modern optimization
and nonlinear analysis; see, e.g., the books
\cite{BC2011}, 
\cite{BorVanBook},
\cite{Brezis}, 
\cite{BurIus},
\cite{Simons1},
\cite{Simons2},
\cite{Zeidler2a},
\cite{Zeidler2b},
and
\cite{Zeidler1}. 

Let $A:X\rras X$ be maximally monotone
and let $\Id:X\to X$ be the identity operator. 
The \emph{resolvent} of $A$ is $J_A:=(\Id+A)^{-1}$
and the \emph{reflected resolvent} is
$R_A:=2J_A-\Id$. It is well-known that
$J_A$ is single-valued, maximally 
monotone and 
firmly nonexpansive. 

The sum problem for two maximally 
monotone operators
$A$ and $B$
is to find $x\in X$ such that $0\in Ax+Bx$.
When $(A+B)^{-1}(0)\neq \fady $
one approach to solve the problem 
is the Douglas--Rachford splitting technique.
Recall that the Douglas--Rachford 
splitting operator 
\cite{L-M79}  for the ordered pair of
operators $(A,B)$ is defined by
\begin{empheq}[box=\mybluebox]{equation}
\label{def:T}
\TAB:=\tfrac{1}{2}(\Id+R_BR_A)
=\Id-J_A+J_B R_A.
\end{empheq} 
Let $x_0\in X$. When $(A+B)^{-1}(0)\neq \fady $
the 
\emph{``governing sequence"} 
$(\TAB^n x_0)_{n\in \NN}$
produced by the Douglas--Rachford operator
 converges 
weakly to a point in $\fix \TAB$
\footnote{$\fix T=\menge{x\in X}{x=Tx}$ is the set of fixed points of $T$.}
(see \cite{L-M79})
and the \emph{``shadow sequence"} 
$(J_A\TAB ^n x_0)_{n\in \NN}$
converges weakly to a point in
$(A+B)^{-1}0$. 
For further information on the Douglas--Rachford algorithm, we
refer the reader to
\cite{EckBer}, 
\cite{L-M79},
\cite{Svaiter2012},
and also \cite{BC2011}. 

When $A:=N_U$ and $B:=N_V$
\footnote{Throughout the paper we use 
$N_C$ and $P_C$ to denote 
the \emph{normal cone} and \emph{projector} associated
with a nonempty closed convex subset $C$ of $X$, respectively. 
}, 
where $U$ and $V$ are 
nonempty closed convex subsets of $X$, 
the sum problem 
is equivalent to the convex 
feasibility problem:
Find $x\in U\cap V$. 
In this case, using \cite[Example~23.4]{BC2011},
\begin{empheq}[box=\mybluebox]{equation}
T:=T_{(N_U,N_V)}=\Id-P_U+P_VR_U,
\end{empheq} 
where $R_U=2P_U-\Id$.
In the inconsistent case, when $U\cap V= \fady$, 
the governing sequence is proved 
to satisfy 
that $\norm{T^n x}\to +\infty$ and the shadow sequence 
$(P_UT^n x)_{n\in \NN}$ remains bounded with 
the weak cluster points being the best approximation 
pairs relative to $U$ and $V$ provided they exist (see \cite{BCL04}).

Unlike the 
method of alternating projections, 
which employs the operator $P_VP_U$,
the Douglas--Rachford method is not fully understood 
in the inconsistent case. 
Nonetheless, the Douglas-Rachford operator
is used in \cite{Sicon2014} to define the ``\emph{normal 
problem}" when the original problem is possibly inconsistent. 
In this case the set of best approximation solutions relative  to
$U$ (which are also known as the \emph{normal solutions}, see \cite{Sicon2014})
is
$U\cap(\gap+V)$, where $\gap=P_{\overline{\ran}(\Id-T)}0$.
It is natural to ask what can we learn about the 
algorithm in the highlight of the new 
concept of the normal problem.

The goal of this paper is to study the case 
 when $U$ and 
$V$ are closed affine subspaces 
that do not necessarily intersect.
The Douglas--Rachford method 
for two closed affine subspaces 
has recently shown to be very useful 
in many applications, for instance, 
the nonconvex sparse affine feasibility problem
(see \cite{HL13} and \cite{HLN14})
and basis pursuit problem (see \cite{DZ14}).
Our results show that
the shadow sequence 
will always converge strongly to a best approximation
solution in $U\cap(\gap+V)$ and therefore we generalize the
main results in \cite{JAT2014}. This is remarkable 
because we do not have to have prior knowledge about the 
\emph{gap vector $\gap$}; 
the shadow sequence is simply $(P_UT^n x_0)_\nnn$.
Our proofs critically rely on the well-developed results 
in the consistent case in \cite{JAT2014}
and the structure of the normal problem 
studied in \cite{Sicon2014}.

We are now ready to briefly summarize our main results: 
\begin{itemize}
\item[{\bf R1}]\namedlabel{R:1}{\bf R1}
We compare the sequences $((_{-\gap}T)^nx)_{n\in \NN}$,
$((T_{-\gap})^nx)_{n\in \NN}$
\footnote{Let $w\in X$. 
We define the \emph{inner shift}
and \emph{outer shift} of an operator $T$
 by $w$  at $x\in X$ by
$
T_{w}x:=T(x-w)$ and $_{w}Tx:=-w+Tx,
$
respectively.
}
 and $(T^nx+n\gap)_{n\in \NN}$
 when $T$ is an affine nonexpansive
 operator
\footnote{Recall that $T$ is \emph{nonexpansive} if
$(\forall x\in X)(\forall y\in X)$
$\|Tx-Ty\|\leq \|x-y\|$.}
and $\gap:=P_{\overline{\Id-T}}0\in\ran(\Id-T)$
\footnote{In highlight of \cref{F:v:WD}
the vector $\gap $ is unique and well-defined.}.
We prove that 
the three sequences coincide
(see \cref{P:aff}).
Surprisingly, when we drop the assumption of $T$ being affine, 
the sequences can be dramatically 
different (see \cref{e:asym}).
\item[{\bf R2}]\namedlabel{R:2}{\bf R2}
We prove the strong convergence 
of the shadow sequence $(P_U T^n x_0)_{n\in \NN}$
 when $U$ and $V$ are affine subspaces 
 that do not have to intersect
 (see \cref{thm:main:aff}).  We identify the limit to be 
 the best approximation 
solution; moreover, the rate of convergence is linear 
when $U+V$ is closed.
\item[{\bf R3}]\namedlabel{R:3}{\bf R3}
 In view of \ref{R:2}
 it is tempting to conjecture that 
 the shadow sequence $(J_A T^n x_0)_{n\in \NN}$
 in the inconsistent case (i.e., when $(A+B)^{-1}=\fady$) 
 converges 
 in a more general setting. We illustrate
 the somewhat surprising fact that 
if $A$ and $B$ are affine --- but
 not normal cone --- operators (see \cref{ex:not:NC}),
 the sequence $(J_A T^n x_0)_{n\in \NN}$
can be unbounded. In fact, we can have 
 $\norm{J_AT^n x_0}\to +\infty$ even though the 
 sum problem has \emph{normal solutions}
 \footnote{The normal solutions are the counterpart 
 of the best approximation solutions in the context
 of the normal problem \cite{Sicon2014} when 
 the operators are not normal cone operators 
 (see \cref{S:main} for details).}. 
 This illustrates that 
 normal cone operators have additional structure that 
 makes \ref{R:2} possible.
\end{itemize}

 \subsection*{Organization}
The remainder of this paper 
is organized as follows. 
\cref{nexp:f:nexp} contains 
a collection of new results 
concerning nonexpansive and 
firmly nonexpansive operators 
whose fixed point sets could 
possibly be empty. 
\cref{S:aff:case} focuses on {affine} 
nonexpansive operators
and their corresponding inner and outer
{``normal" shifts}. 
Various examples that illustrate 
our theory are provided. 
\cref{S:main} is devoted to 
present the main results.
We prove 
strong convergence of the shadows 
of the Douglas-Rachford iterates
of two (not necessarily intersecting) 
affine subspaces. 
\subsection*{Notation}
Let $C$ be a nonempty closed convex 
subset of $X$. 
The \emph{recession cone} of $C$ is
$\rec C:=\{x\in X~|~x+C\subseteq C\}$,
the \emph{polar cone} of $C$
is $C^{\ominus}:=\{u\in X~|~\sup \innp{C,U}\le 0\}$
and the \emph{dual cone} of $C$
is $C^\oplus=-C^\ominus$.
When $C$ is an affine subspace 
the \emph{linear space parallel to $C$} is
$\parl C=C-C$.
Otherwise, the notation we utilize is standard and follows, 
e.g., \cite{BC2011} and \cite{Rock98}.

\section{Nonexpansive and firmly nonexpansive operators}
\label{nexp:f:nexp}
In this section, we collect various results on (firmly)
nonexpansive operators that will be useful later. 
Let $w\in X$. Recall  that for a single-valued or set-valued 
operator $T$ we define the \emph{inner shift}
and \emph{outer shift} by $w$  at $x\in X$ by
\begin{equation}\label{eq:def:in:out}
T_{w}x:=T(x-w) \quad \text{and}\quad _{w}Tx:=-w+Tx,
\end{equation}
respectively.
\begin{lemma}\label{lem:in:out:T}
Let $T:X\to X$ and let $w\in X$. Then the following hold:
\begin{enumerate}
\item\label{lem:in:out:T:i}
$\fix(T_{-w})=-w+\fix(w+T)=-w+\fix(_{-w}T)$.
\item\label{lem:in:out:T:i:ii}
$w\in \ran(\Id-T)\iff \fix (w+T)
\neq \fady\iff \fix(T_{-w})\neq \fady$.
\item\label{lem:in:out:T:ii}
$(\forall x\in X)(\forall n\in \NN)
\quad(T_{-w})^n x=-w+(w+T)^n(x+w)
=-w+(_{-w}T)^n(x+w)$.
\end{enumerate}
\end{lemma}
\begin{proof}
\ref{lem:in:out:T:i}: 
Let $x\in X$. Then $x\in \fix(T_{-w})$
$\iff x=T(x+w)\iff x+w =w+T(x+w)\iff x+w\in \fix(w+T)$
$\iff x\in -w+ \fix(w+T)$.

\ref{lem:in:out:T:i:ii}:
$w\in \ran(\Id-T)$ $\iff (\exists x\in X)$ such that
$w=x-Tx$ $\iff(\exists x\in X)$ such that $x=w+Tx$ 
$\iff \fix(w+T)\neq\fady$. Now combine with 
\ref{lem:in:out:T:i}.

\ref{lem:in:out:T:ii}: We proceed by induction. 
The conclusion is clear 
when $n=0$. Now assume 
that for some $n\in \NN$
it holds that $(T_{-w})^n x=-w+(w+T)^n(x+w)$. Then 
$(T_{-w})^{n+1}  x
=T((T_{-w})^n x+w)=T(-w+(w+T)^n(x+w)+w)=-w+w+T((w+T)^n(x+w))
=-w+(w+T)^{n+1}(x+w)$, as claimed.
\end{proof}

We recall the following important fact.
\begin{fact}[\bf{Infimal displacement vector}]
\label{F:v:WD}
{\rm (See, e.g., \cite{Ba-Br-Reich78},\cite{Br-Reich77} 
and \cite{Pazy}.)}
Let $T:X\to X$ be nonexpansive.
Then $\overline{\ran}(\Id-T)$ is convex;
consequently, the infimal displacement vector
\begin{equation}
\gap:=P_{\overline{\ran}{(\Id-T)}}0 
\end{equation}
is the unique and well-defined 
element in $\overline{\ran}{(\Id-T)}$
such that
$
\norm{\gap}=\ds\inf_{x\in X}\norm{x-Tx}.
$
\end{fact}
Unless stated otherwise, throughout this paper 
we assume that 
\begin{empheq}[box=\mybluebox]{equation}\label{T:assmp}
T \text{~~is a nonexpansive operator on X},
\end{empheq} 
and that
\begin{empheq}[box=\mybluebox]{equation}\label{gap:assmp:pre}
\gap:=P_{\overline{\ran}{(\Id-T)}}0\in {{\ran}\bk{\Id-T}}.
\end{empheq} 
In view of 
\cref{gap:assmp:pre} and
\cref{lem:in:out:T}\ref{lem:in:out:T:i:ii}
we have 
\begin{equation}\label{e:not:fady:fix}
\fix(T_{-\gap})\neq \fady
\quad
\text{and}
\quad
 \fix(\gap+T)\neq \fady.
\end{equation}
We start with the following useful result.
\begin{lem}\label{lem:norm:proj}
Let $C$ be a nonempty closed convex subset of $X$
and let $c\in C$ satisfies that $\norm{c}=\norm{P_C0}$.
Then $c=P_C0$.
\end{lem}
\begin{proof}
See \nameref{Appp:1}.
\end{proof}
\begin{prop}\label{prop:fix:inc}
Let $y_0\in \fix(\gap+T)$.
Then the following hold:
\begin{enumerate}
\item\label{prop:fix:inc:i}
$y_0-\RR_{+}\gap\subseteq \fix(\gap+T)$. 
\item\label{prop:fix:inc:-1}
$\fix(\gap+T)-\RR_+\gap=\fix(\gap+T)$.
\item\label{prop:fix:inc:i:0}
$-\RR_+\gap\subseteq \rec(\fix(\gap+T))$.
\item\label{prop:fix:inc:i:i}
$(\forall n\in \NN)$ $T^n y_0=y_0-n\gap$.
\item\label{prop:fix:inc:ii}
$\left]-\infty,1\right]\cdot\gap+\fix T_{-\gap}\subseteq\fix (\gap+T)$.
In particular it holds that $\fix (T_{-\gap}) \subseteq \fix(\gap+T)$.

\item\label{prop:fix:inc:ii:ii}
For every $x\in X$, the sequence $(T^nx+n\gap)_{n\in \NN}$
is F\'{e}jer monotone with respect to both $\fix(\gap+T)$ 
and $\fix (T_{-\gap})$.
\item\label{prop:fix:inc:iii} 
Suppose that $x_0\in \fix T_{-\gap}$ and
set $(\forall n\in \NN )~x_n=T^nx_0$. 
Then  $x_n=x_0-n\gap$
 and $(x_n)_{n\in \NN}
$ lies in $ \fix(T_{-\gap})$.
\end{enumerate}
\end{prop}
\begin{proof}
\ref{prop:fix:inc:i}:
First we use induction to show that
\begin{equation}\label{eq:n:induc}
(\forall n\in \NN) \quad y_0-n\gap\in \fix (\gap+T).
\end{equation}
Clearly when $n=0$ the base case holds true.
Now suppose that for some $n\in \NN$
it holds that
$y_0-n\gap\in \fix (\gap+T)$, i.e.,
\begin{equation}\label{eq:ax:i}
y_0-n\gap=\gap+T(y_0-n\gap).
\end{equation}

Using \cref{gap:assmp:pre} and \cref{eq:ax:i} we have
\begin{align*}
\norm{\gap}&\le\norm{(\Id-T)(y_0-(n+1)\gap)}
=\norm{y_0-(n+1)\gap-T(y_0-(n+1)\gap)}\\
&=\norm{y_0-n\gap -\gap-T(y_0-(n+1)\gap)}
=\norm{T(y_0-n\gap )-T(y_0-(n+1)\gap)}
\le \norm{\gap}.
\end{align*}
Consequently all the inequalities above are
 equalities
and we conclude that 
$\norm{\gap}=\norm{y_0-(n+1)\gap-T(y_0-(n+1)\gap)}$. 
It follows from \cref{gap:assmp:pre} and 
\cref{lem:norm:proj} that 
\begin{equation}
y_0-(n+1)\gap-T(y_0-(n+1)\gap)=\gap.
\end{equation}
That is, $y_0-(n+1)\gap=\gap+T(y_0-(n+1)\gap)$,
which proves \cref{eq:n:induc}.
Now using \cite[Corollary~4.15]{BC2011}
we learn that $\fix(\gap+T)$ is convex, which when combined with
\cref{eq:n:induc} yields \ref{prop:fix:inc:i}.

\ref{prop:fix:inc:-1}:
On the one hand it follows from \ref{prop:fix:inc:i} that
$\fix(\gap+T)-\RR_+\gap\subseteq\fix(\gap+T)$.
On the other hand
$\fix(\gap+T)=\fix(\gap+T)-0
\cdot\gap\subseteq \fix(\gap+T)-\RR_+\gap$.

\ref{prop:fix:inc:i:0}: This follows directly from 
\ref{prop:fix:inc:-1}.

\ref{prop:fix:inc:i:i}:
We use induction. Clearly $y_0-0\gap=y_0=T^0y_0$.
Now suppose that for some $n\in \NN$ it holds 
$T^ny_0=y_0-n\gap$. 
Using \ref{prop:fix:inc:i} we have
$T^{n+1}y_0=T(y_0-n\gap)=-\gap+y_0-n\gap
=y_0-(n+1)\gap$.

\ref{prop:fix:inc:ii}:
Using \cref{lem:in:out:T}\ref{lem:in:out:T:i} 
and \ref{prop:fix:inc:i} we have
$\left]-\infty,1\right]\cdot\gap+\fix T_{-\gap}
=\gap-\RR_+\gap+\fix T_{-\gap}
=\fix(\gap+T)-\RR_+\gap=\fix(\gap+T)$.
In particular we have 
$\fix T_{-\gap}= 0\cdot\gap+\fix T_{-\gap}
\subseteq\fix(\gap+T)$.

\ref{prop:fix:inc:ii:ii}:
Let $x\in X$ and let $y\in \fix(\gap+T)$. 
Then using \ref{prop:fix:inc:i:i}
we have for every $n\in \NN$,
\begin{align*}
\norm{T^{n+1}x+(n+1)\gap-y}
&=\norm{T^{n+1}x-(y-(n+1)\gap)}
=\norm{T^{n+1}x-T^{n+1}y}\\
&\le \norm{T^{n}x-T^{n}y}=\norm{T^{n}x-(y-n\gap)}
=\norm{T^{n}x+n\gap-y}.
\end{align*}
The statement for $\fix T_{-\gap}$
follows from \ref{prop:fix:inc:ii}.

\ref{prop:fix:inc:iii}:
Combine \ref{prop:fix:inc:ii} and \ref{prop:fix:inc:i:i}
to get that $x_n=x_0-n\gap$.
Now by 
\cref{lem:in:out:T}\ref{lem:in:out:T:i} 
$x_0+\gap\in \fix(\gap+T)$. 
Using \ref{prop:fix:inc:i} we have 
$(\forall n\in \NN)$ $x_0+\gap-n\gap\in \fix(\gap+T)$
or equivalently by \cref{lem:in:out:T}\ref{lem:in:out:T:i}
$x_0-n\gap\in -\gap+\fix(\gap+T)=\fix(T_{-\gap})$.
\end{proof}

The next example is readily verified.
\begin{example}\label{ex:short}
Let $C$ be a nonempty closed convex subset of $X$
and suppose that $T=\Id-P_C$. Then $T$ is 
firmly nonexpansive\footnote{Recall that $T\colon X\to X$ is \emph{firmly nonexpansive}
if 
$
(\forall x\in X)(\forall y\in X)\quad\|Tx-Ty\|^2 + \|(\Id-T)x-(\Id-T)y\|^2\leq \|x-y\|^2.
$
}
and $\gap=P_C0$.
Let $x\in X$. Then 
 $x\in \fix (\gap +T)\iff P_C x=\gap$, 
while $x\in \fix T_{-\gap}\iff P_C (x+\gap)=\gap$.
\end{example}
 \begin{prop}\label{prop:RR:conv}
Suppose that $X=\RR$,
and that $\fix T= \fady$. 
 Let $x\in \RR$ and set 
 $(\forall n\in \NN)~ y_n=T^nx+n\gap$.
Then the following hold:
\begin{enumerate}
\item\label{prop:RR:conv:i}
$(y_n)_{n\in \NN}$ converges.
\item\label{prop:RR:conv:ii}
$\RR\to\RR:x\mapsto \ds\lim_{n\to \infty} (T^n x+n\gap)$ is nonexpansive.
\item\label{prop:RR:conv:iii}
 Suppose that $T$ is firmly nonexpansive. Then
$\RR\to\RR:x\mapsto \ds\lim_{n\to \infty} (T^n x+n\gap)$ is firmly nonexpansive.
\end{enumerate}
\end{prop}
 \begin{proof}
 \ref{prop:RR:conv:i}:
 In view of
 \cref{prop:fix:inc}\ref{prop:fix:inc:ii:ii} 
 the sequence $(y_n)_{n\in \NN}$
    is F\'{e}jer monotone with respect to $\fix (\gap+T)$. Now by 
    \cref{prop:fix:inc}\ref{prop:fix:inc:i} we know that $
    \fix (\gap+T)$ contains an unbounded interval. 
    Since $X=\RR$ we conclude that
    $\intr  \fix (\gap+T) \neq \fady$. It follows from 
    \cite[Proposition~5.10]{BC2011} 
    that $(y_n)_{n\in \NN}$ converges.  
    
    \ref{prop:RR:conv:ii}:
    Let $y\in \RR$. Then 
    \begin{align}
    \Abs{\lim_{n\to \infty} (T^n x+n\gap)-\lim_{n\to \infty} (T^n y+n\gap)}
    &=\Abs{\lim_{n\to \infty} (T^n x+n\gap-T^n y-n\gap)}\nonumber\\
    &=\lim_{n\to \infty} \abs{T^n x-T^n y}
    \le\lim_{n\to \infty} \abs{x-y}= \abs{x-y}.
    \end{align}
     \ref{prop:RR:conv:iii}: It follows from 
    \cite[Proposition~4.2(iv)]{BC2011} that an operator is 
    firmly nonexpansive if and only if it is nonexpansive and 
    monotone.
    Therefore, in view of \ref{prop:RR:conv:ii}, 
    we need to check monotonicity.
    Without loss of generality let $y\in \RR$ such that $x\le y$. 
  Since $T$ is firmly nonexpansive, hence monotone, 
  one can verify that
   $(\forall n\in \NN) $ $T^nx\le T^ny$ and therefore
    $(\forall n\in \NN) $ $T^nx+n\gap\le T^ny+n\gap$.
    Now take the limit as $n\to \infty$. 
    \end{proof}
 
When $X=\RR$, 
it follows from \cref{prop:RR:conv}\ref{prop:RR:conv:i}
that the sequence $(T^nx+n\gap)_{n\in \NN}$
converges.
In view of \cref{prop:fix:inc}\ref{prop:fix:inc:ii:ii} 
the sequence $(T^nx+n\gap)_{n\in \NN}$ is 
F\'ejer monotone with respect to $\fix (\gap+T)$ 
which might suggest that the limit lies in 
$\fix (\gap+T)$. We show in the following example
 that this is not true in general.

 \begin{ex}\label{exmp:not:fejer}
Suppose that $X=\RR$ and that
\begin{equation}\label{ex:not:fejer}
T:\RR\to \RR
:x\mapsto
\begin{cases}
x-\alpha,&\text{~\rm if~}x\le \alpha;\\
0,&\text{~\rm if~}\alpha<x\le\beta;\\
x-\beta,&\text{~\rm if~}x> \beta,
\end{cases}
\end{equation}
 where $0<\alpha<\beta$.
 Then $T$ is firmly nonexpansive but \emph{not affine},
 $\gap=\alpha$, $\fix(\gap+T)=\left ]-\infty,\alpha\right]$,
$\fix T_{-\gap}=\left ]-\infty,0\right],$
and 
\begin{equation}\label{e:all:cases}
T^n+n\gap:\RR\to\RR:x\mapsto
\begin{cases}
x,&\text{~\rm if~}x\le \alpha;\\
\alpha,&\text{~\rm if~}\alpha<x\le \beta;\\
x-n(\beta-\alpha),
&\text{~\rm if~}x>\beta 
\text{~\rm and~} n\le \lfloor x/\beta \rfloor;\\
\min\stb{\alpha,x-\Big\lfloor  \tfrac{x}{\beta} \Big\rfloor\beta}
+\Big\lfloor  \tfrac{x}{\beta} \Big\rfloor\alpha,
&\text{~\rm if~}x>\beta \text{~\rm and~} 
n> \lfloor x/\beta \rfloor.
\end{cases}
\end{equation}
Consequently,
\begin{equation}\label{e:all:cases:lim}
\lim_{n\to \infty}(T^n+n\gap):\RR\to\RR:x\mapsto
\begin{cases}
x,&\text{~\rm if~}x\le \alpha;\\
\alpha,&\text{~\rm if~}\alpha<x\le \beta;\\
\min\stb{\alpha,x-\Big\lfloor  \tfrac{x}{\beta} \Big\rfloor \beta}
+\alpha\Big\lfloor \tfrac{x}{\beta} \Big\rfloor,
&\text{~\rm if~}x>\beta .
\end{cases}
\end{equation}
Therefore
for every $ x_0\in \RR$ the sequence 
$(T^nx_0+n\gap)_{n\in \NN}$ 
is eventually constant. 
However, 
if the starting point $x_0$ lies in the interval  ~$\left]\beta,\infty\right[$,
  then
$ \lim_{n\to \infty} 
T^nx_0+n\gap=\min\{\alpha,x-\lfloor  \tfrac{x}{\beta} \rfloor \beta\}
+\alpha\lfloor \tfrac{x}{\beta} \rfloor\not\in \fix (\gap+T)$.
\end{ex}
\begin{proof}
See \nameref{app:3}.
\end{proof}

\section{Affine nonexpansive operators}\label{S:aff:case}
In this section, we investigate properties of \emph{affine}
nonexpansive operators. This additional assumption allows for
stronger results than those obtained in the previous section.
We recall the following fact.
\begin{fact}{\rm (See \cite[Proposition~3.17]{BC2011}.)}\label{F:trans}
Let $S$ be a nonempty subset of $X$, and let $y \in X$. 
Then 
\begin{equation}
(\forall x\in X) ~~P_{y+S}x=y+P_S(x-y).
\end{equation}
\end{fact}

\begin{thm}\label{P:aff}
Let $\linop\colon X\to X$ be linear and nonexpansive, 
let $b\in X$, and suppose that
$T\colon X\to X\colon x\mapsto \linop x + b$. 
Suppose also that $\gap\in\ran(\Id-T)$,
and let $x\in X$.
Then the following hold:
\begin{enumerate}
\item\label{P:aff:i}
$\gap=P_{\fix \linop}(-b)\in \fix \linop=(\ran(\Id-\linop))^\perp$,
and $\gap\neq 0\iff b\not\in \ran(\Id-\linop)$.
\item\label{P:aff:ii}
$(\forall n\in \NN)$ 
$T^nx=\linop^nx+\sum_{k=0}^{n-1}\linop^kb$.
\item\label{P:aff:iii}
$(\forall n\in \NN)$ 
$T^nx+n\gap=\linop^nx+\sum_{k=0}^{n-1}\linop^kP_{\ran(\Id-\linop)}b$.
\item\label{P:aff:iv}
$(\forall n\in \NN)$ 
$(T_{-\gap})^{n}x=T^nx+n\gap$.
\item\label{P:aff:v}
$(\forall n\in \NN)$ 
$(T_{-\gap})^{n}x=(\gap+T)^nx$.
\item\label{P:aff:vi}
$\fix T_{-\gap}=-\gap+\fix T_{-\gap}
=-\gap+\fix (\gap+T)=\fix (\gap+T)$.
\item\label{P:aff:vii}
$\fix(T_{-\gap})=\fix(\gap+T)=\RR\gap+\fix(\gap+T)
=\RR\gap+\fix(T_{-\gap})$.
Consequently $\gap$ lies in the \emph{lineality space}
\footnote{For the definition and 
a detailed discussion of the 
lineality space, we refer the 
reader to \cite[page 65]{Rock70}.}
of 
$\fix (T_{-\gap})=\fix (\gap+T)$. 
\end{enumerate}
\end{thm}
\begin{proof}
\ref{P:aff:i}:
Note that $\ran(\Id-T)=\ran (\Id-L)-b$
and hence
$\overline{\ran}(\Id-T)=\overline{\ran}(\Id-L)-b$. 
Therefore,
using \cref{F:trans} we have
$\gap=P_{\overline{\ran}(\Id-T)}0
=P_{-b+\overline{\ran}(\Id-L)}0=-b+P_{\overline{\ran}(\Id-L)}(0-(-b))
=-b+P_{\overline{\ran}(\Id-\linop)}b
$.
Using 
\cite[Fact~2.18(iv)]{BC2011} and 
\cite[Lemma~2.1]{BDHP03}, 
we learn that 
$\overline{\ran}(\Id-\linop)^\perp
=\ker(\Id-\linop^*)=\fix L^*=\fix \linop$,
and hence  
\begin{equation}
\gap
=(\Id-P_{\overline{\ran}(\Id-\linop)})(-b)
=P_{(\overline{\ran}(\Id-\linop))^{\perp}}(-b)
=P_{\fix \linop}(-b).
\end{equation}
Note that $\gap\neq 0\iff b\not\in \overline{\ran}(\Id-\linop)$.

\ref{P:aff:ii}:
We prove this by induction. 
When $n=0$ the conclusion is obviously true.
Now suppose that for some $n\in \NN$
it holds that 
\begin{equation}
T^nx=\linop^n x+\sum_{k=0}^{n-1}\linop^kb.
\end{equation}  
Then 
$T^{n+1}x=T(T^nx)
=T(\linop^n x+\sum_{k=0}^{n-1}\linop^kb)= 
\linop(\linop^n x+\sum_{k=0}^{n-1}\linop^kb)+b
=\linop^{n+1} x+\sum_{k=0}^{n}\linop^kb$,
as claimed.

\ref{P:aff:iii}: Note that 
$b=P_{{\overline{\ran}}(\Id-\linop)}b+P_{\fix \linop}b$.
Using \ref{P:aff:i} and \ref{P:aff:ii} yields
\begin{align*}
T^nx+n\gap&=\linop^n x+\sum_{k=0}^{n-1}(\linop^kb+\gap)
=\linop^n x+\sum_{k=0}^{n-1}\bk{\linop^kb+\linop^k \gap}\\
&=\linop^n x+\sum_{k=0}^{n-1}\bk{\linop^kb-\linop^kP_{\fix \linop}b}
=\linop^n x+\sum_{k=0}^{n-1}\linop^k(\Id-P_{\fix \linop})b\\
&=\linop^n x+\sum_{k=0}^{n-1}\linop^kP_{\overline{\ran}(\Id-\linop)}b.
\end{align*}

\ref{P:aff:iv}: 
We prove this by induction. Note that 
by \ref{P:aff:i} $\gap\in \fix \linop$, hence 
$\linop\gap=\gap$.
When $n=0$ we have $(T_{-\gap})^0 x=x=
T^0x+0\cdot\gap$.
Now suppose that for some $n\in \NN$
it holds that 
$(T_{-\gap})^{n}x=T^nx+n\gap$.
Then $(T_{-\gap})^{n+1}x
=T_{-\gap}(T^nx+n\gap)=T(T^nx+n\gap+\gap)
=\linop(T^nx)+\linop((n+1)\gap)+b
=T^{n+1}x+(n+1)\gap$.

\ref{P:aff:v} We use induction again.
The base case is obviously true. 
Now suppose that for some $n\in \NN$
it holds that 
$
(\gap+T)^{n}x=T^nx+n\gap.
$
Then $(\gap+T)^{n+1}x=\gap+T(\gap+T)^{n}x
=\gap+T(T^nx+n\gap)$
$=\gap+L(T^nx+n\gap)+b
=\gap+LT^nx+n\gap+b=LT^nx+b+(n+1)\gap$
$=T^{n+1}x+(n+1)\gap$.
Now combine with \ref{P:aff:iv}.

\ref{P:aff:vi}:
Using \ref{P:aff:v} with $n=1$ we have $T_{-\gap}=\gap+T$.
Now apply \cref{lem:in:out:T}\ref{lem:in:out:T:i}.

\ref{P:aff:vii}:
Using \ref{P:aff:vi} 
and the assumption that $T$ is an affine operator,
we have $\fix(T_{-\gap})=\fix(\gap+T)$ is an affine subspace. 
Now let $y_0\in \fix(T_{-\gap})=\fix(\gap+T)$.
Using \cref{prop:fix:inc}\ref{prop:fix:inc:i}
we have $-\RR_{+}\gap\subseteq \fix (\gap +T)-y_0
=\parl  \fix (\gap +T)$ and therefore 
$\RR\gap \subseteq \parl \fix (\gap +T)$.
Hence $y_0+\RR\gap\subseteq \fix (\gap +T)$
which yields $\fix (\gap +T)+\RR\gap\subseteq \fix (\gap +T)$.
Since the opposite inclusion is obviously true we conclude 
that  
\ref{P:aff:vii} holds.
\end{proof}

Suppose $T$ is nonexpansive but not affine.
\cref{P:aff} might suggest that, for every $x\in X$,
 the sequences $\bk{\T^n x+n\gap}_{n\in\NN}$, 
 $\bk{T_{-\gap}^{n}x}_{n\in\NN}$ and $\bk{(\gap+T)^n x}_{n\in\NN}$ 
 coincide, and consequently $\bk{\T^n x+n\gap}_{n\in\NN}$
 is a sequence of iterates of a nonexpansive operator.
 Interestingly, this is not the case as we illustrate now. 
\begin{ex}\label{e:asym}
Suppose that $X=\RR$ and let $\beta>0$. 
Suppose that 
\begin{equation}\label{eq:ex:long}
\T\colon\RR\to\RR: x\mapsto\begin{cases}
x-\beta,&x\le\beta;\\
\alpha\bk{x-\beta},& x>\beta,
\end{cases}
\end{equation}
where $0<\alpha<1$.
Then $\fix T=\fady$, $\gap=\beta$,
for every $ n\in \NN$
\begin{equation}\label{ex:inn:sft}
\bk{T_{-\gap}}^n :\RR\to \RR: x\mapsto
\alpha^n\max\stb{x,0}+\min\stb{x,0},
\end{equation}
\begin{equation}\label{ex:out:sft}
(\gap+T)^n :\RR\to \RR: x\mapsto
\alpha^n\max\stb{x-\beta,0}+\min\stb{x,\beta},
\end{equation}
and
\begin{equation}\label{ex:nv:sft}
\T^n +n\gap:\RR\to \RR: x\mapsto\begin{cases}
x,&\text{~\rm if~}x\le \beta;\\
\alpha^n x-\bk{\frac{\alpha(1-\alpha^{n})}{1-\alpha}}\beta
+n\beta,&\text{~\rm if~} x>\beta, n< q(x);\\
\alpha^{q(x)} x-\bk{\frac{\alpha(1-\alpha^{q(x)})}{1-\alpha}}\beta
+q(x)\beta,&\text{~\rm if~} x>\beta, n\ge q(x),
\end{cases}
\end{equation}
where $q(x):\RR\to \NN
:x\mapsto\left\lceil {\log_\alpha \frac{\beta}
{\alpha\beta+(1-\alpha)x}}\right \rceil $.
Consequently,
\begin{equation}\label{ex:inn:sft:lim}
(\forall x\in \RR)\quad \lim_{n\to \infty}\bk{T_{-\gap}}^nx=
\min\stb{x,0},
\end{equation}
\begin{equation}\label{ex:out:sft:lim}
(\forall x\in \RR)\quad \lim_{n\to \infty}(\gap+T)^{n}  x
=\min\stb{x,\beta},
\end{equation}
and
 \begin{equation}\label{ex:nv:sft:lim}
(\forall x\in \RR)\quad
\lim_{n\to \infty}\bk{T^n x+n\gap} =\begin{cases}
x,&\text{~\rm if~}x\le\beta;\\
\alpha^{q(x) }x-
\bk{\frac{\alpha(1-\alpha^{q(x)})}{1-\alpha}}\beta
+q(x)\beta,&\text{~\rm if~} x>\beta.
\end{cases}
\end{equation}
Moreover, there is no operator $S:\RR\to \RR$
such that for every $x\in \RR$  and for every $ n\in \NN$ 
 we have $S^n x=T^n x+n\gap$.
 \end{ex}
  \begin{proof}
  See \nameref{app:D}.
  \end{proof}
\begin{figure}[H]
 \centering
 \includegraphics[scale=0.23]{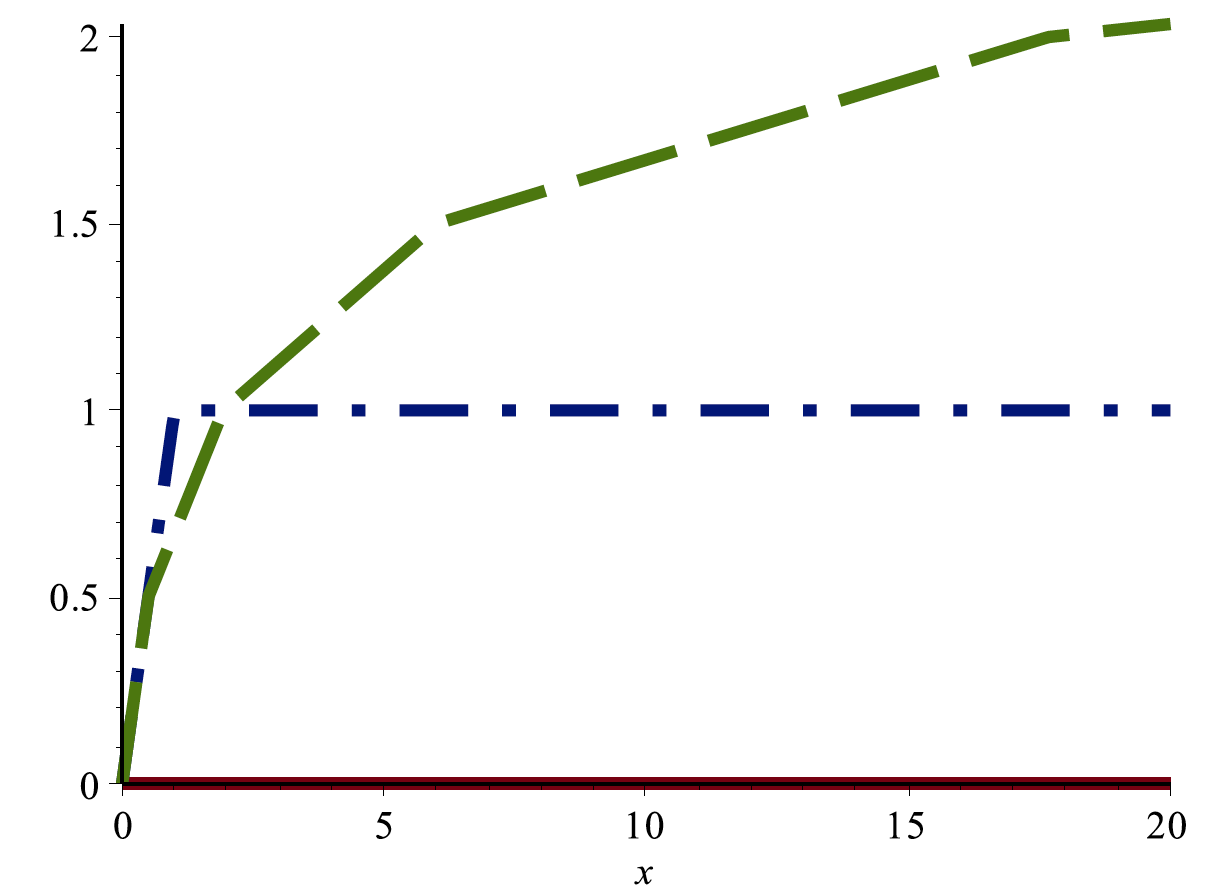}
 \caption{The solid curve represents 
 $\lim_{n\to \infty}(T_{-\gap})^n x$,
the dashed dotted curve represents 
$\lim_{n\to \infty}(\gap+T)^n x$,
and the dashed curve represents 
$\lim_{n\to \infty}T^nx+n\gap$, 
when $\alpha=0.5$ and $\beta=1$.}
\label{fig:com:lim}
\end{figure}
Figure \ref{fig:com:lim} provides a plot of the functions 
defined by \cref{ex:inn:sft:lim}, 
\cref{ex:out:sft:lim} and \cref{ex:nv:sft:lim} that 
illustrates that they are pairwise distinct.
\section{The Douglas--Rachford operator for two affine subspaces}
\label{S:main}
 Unless otherwise stated we assume 
from now on that
\begin{empheq}[box=\mybluebox]{equation*}
A \text{~~and~~} B \text{~~are maximally 
monotone operators on~~}X.
\end{empheq} 
The Attouch--Th\'{e}ra dual pair of $(A,B)$
(see \cite{AT})
is the pair $(A,B)^*:=(A^{-1},B^{-\ovee})$,
where 
\begin{equation}
A^{\ovee}:= (-\Id)\circ A\circ(-\Id)\quad{\text{and}}
\quad A^{-\ovee}:=(A^{-1})^\ovee=(A^\ovee)^{-1}.
\end{equation}
We shall use
\begin{equation}
Z:=Z_{(A,B)}=(A+B)^{-1}(0)
\qquad\text{and }\qquad 
K:=K_{(A,B)}=(A^{-1}+B^{-\ovee})^{-1}(0),
\end{equation}
to denote the primal and dual solutions respectively
(see e.g. \cite{JAT2012}).

The \emph{normal problem} associated 
with the ordered pair 
$(A,B)$ (see \cite{Sicon2014}) is to find $x\in X$
such that
\begin{equation}
0\in {_{\gap}A} x+B_{\gap}x=Ax-\gap+B(x-\gap), 
\end{equation}
{where} 
\begin{empheq}[box=\mybluebox]{equation}
\gap=P_{\overline{\ran}(\Id-T)},
\end{empheq}
and $T=\TAB$ is defined by \cref{def:T}.
We recall 
(see \cite[Lemma~2.6(iii)]{Comb04} and \cite[Corollary~4.9]{JAT2012})
that
\begin{equation}\label{Fact:collect:Z:fix}
Z=J_A(\fix T) \quad\text{and}\quad K=(\Id-J_A)(\fix T),
\end{equation}
and that $T$ is self-dual 
(see \cite[Lemma 3.6 on page 133]{EckThesis}
and \cite[Corollary~4.3]{JAT2012}), 
i.e.,  
\begin{equation}\label{Fact:collect:T:self}
T_{(A,B)}=T_{(A,B)^*}
=T_{(A^{-1},B^{-\ovee})}.
\end{equation}
The \emph{normal pair} associated with the ordered pair $(A,B)$
is the pair $(_{\gap}A, B_{\gap})$ 
and the \emph{normal Douglas-Rachford operator} 
is $T_{\bk{_{\gap}A, B_{\gap}}}$. Using  
\cite[Proposition~2.24]{Sicon2014}
we have
\begin{empheq}[box=\mybluebox]{equation}
\label{Fact:collect:normal:T}
T_{(_{\gap}A, B_{\gap})}=T_{-\gap}.
\end{empheq}  
The set of \emph{normal solutions} 
is $Z_{\gap}:=Z_{\bk{_{\gap}A, B_{\gap}}}$ 
and 
the set of \emph{dual normal solutions} 
is $K_{\gap}:=K_{\bk{_{\gap}A, B_{\gap}}}$. 
\begin{lem}\label{Fact:collect}
The following hold:
\begin{enumerate}
\item\label{Fact:collect:normal:Z}
$Z_{\gap}=J_{_{\gap}A}
(\fix (T_{-\gap}))
=J_{(-\gap+A)}(\fix (T_{-\gap}))=J_{A}(\fix (T_{-\gap})+\gap)=J_A(\fix(\gap+ T))$.
\item\label{Fact:collect:normal:K}
$K_{\gap}=(\Id-J_{_{\gap}A})
(\fix (T_{-\gap}))
=(\Id-J_{(-\gap+A)})(\fix (T_{-\gap}))$.
\item\label{Fact:collect:non:fady:Z}
 $K_{\gap}\neq \fady\iff Z_{\gap}\neq \fady\iff \gap\in \ran(\Id-T)$.
\end{enumerate}
\end{lem}
\begin{proof}
\ref{Fact:collect:normal:Z}:
Apply \cref{Fact:collect:Z:fix} to the normal pair
$(_{\gap}A, B_{\gap})$ and use \cref{Fact:collect:normal:T}
and \cref{eq:def:in:out}. 
Now apply 
\cite[Proposition~23.15(ii)]{BC2011}.
The last equality follows from \cref{lem:in:out:T}\ref{lem:in:out:T:i}.
\ref{Fact:collect:normal:K}:
Apply \cref{Fact:collect:Z:fix} to the normal pair
$(_{\gap}A, B_{\gap})$ then use \cref{Fact:collect:normal:T}
and \cite[Proposition~23.15(iii)]{BC2011}. 
\ref{Fact:collect:non:fady:Z}: 
The first equivalence follows from applying 
\cite[Proposition~2.4(v)]{JAT2012} to the normal
pair $(_{\gap}A, B_{\gap})$.
Now combine
\cref{lem:in:out:T}\ref{lem:in:out:T:i:ii} 
and \ref{Fact:collect:normal:Z}.
\end{proof}

In the following we assume that 
 \begin{empheq}[box=\mybluebox]{equation}
\label{gap:assmp}
\gap=P_{\overline{\ran}{(\Id-T)}}0\in {{\ran}\bk{\Id-T}},
\end{empheq}
that 
\begin{empheq}[box=\mybluebox]{equation}\label{assm:U:V:sets}
U \text{~~and~~} V \text{~~are 
nonempty closed convex subsets of~~} X
\end{empheq}
 and that
 \begin{empheq}[box=\mybluebox]{equation}\label{assm:U:V:op}
A=N_U \text{~~and~~} B=N_V.
\end{empheq}
Using \cite[Example~23.4]{BC2011},
\cref{def:T} becomes
\begin{empheq}[box=\mybluebox]{equation}
\label{def:T:cones}
\DRS{U}{V}:=T_{(N_U,N_V)}=\Id-P_U+P_VR_U,
\end{empheq} 
where $R_U=2P_U-\Id$.
In this case (see \cite[Proposition~3.16]{Sicon2014})
 \begin{equation}
\gap=P_{\overline{U-V}}0,
\end{equation}
or equivalently 
\begin{equation}\label{rec:v}
-\gap\in N_{\overline{U-V}}(\gap).
\end{equation}
The normal problem now is to find $x\in X$
such that
\begin{equation}
0\in N_Ux-\gap+N_V(x-\gap).
\end{equation}

\begin{lem}\label{l:nor:shf}
Let $ w\in X$. Then the following hold:
\begin{enumerate}
\item\label{l:nor:shf:iii} 
$ J_{-w+N_U}=J_{N_U}(\cdot+w)
=P_U(\cdot+w)$.
\item\label{l:nor:shf:ii} 
$ J_{N_U(\cdot-w)}=w+J_{N_U}(\cdot-w)
=w+P_U(\cdot-w)$.
\item\label{l:nor:shf:i}
 $N_V(\cdot-w)=N_{w+V} $.
\item\label{l:nor:shf:iv}
Suppose that $U$ is an affine subspace and and that $w\in (\parl U)^\perp$.
Then $(\forall \alpha\in \RR)$$(\forall x\in X)$
$P_U (x+\alpha w)=
P_Ux$.
\end{enumerate}
\end{lem}
\begin{proof}
 \ref{l:nor:shf:iii}
and\ref{l:nor:shf:ii}: 
See \cite[Proposition~23.15(ii) and (iii) and Example~23.4]{BC2011}.
\ref{l:nor:shf:i}: One can easily verify that 
$(\forall w\in X)$ we have $\iota_{V}(\cdot-w)=\iota_{w+V}$.
Therefore $N_V(\cdot-w)=\pt \iota_{V}(\cdot-w)
=\pt \iota_{w+V}=N_{w+V}$.
\ref{l:nor:shf:iv}: Let $a\in U$. 
Then $U=a+\parl U$. 
Using \cref{F:trans}, we have
$(\forall \alpha\in \RR)$$(\forall x\in X)$
$P_U (x+\alpha w)=P_{a+\parl U} (x+\alpha w)
=a+P_{\parl U} (x+\alpha w-a)
=a+P_{\parl U} (x-a)+\alpha P_{\parl U} w
=a+P_{\parl U} (x-a)=P_{a+\parl U }x=P_Ux$.
\end{proof}

\begin{prop}\label{aff:sh:con}
Suppose that $U$ and $V$ are closed affine subspaces
of $X$ and  
that $T=\DRS{U}{V}$. 
Then the following hold:
\begin{enumerate}
\item\label{aff:sh:con:i}
T is affine and $T=\Id-P_U-P_V+2P_VP_U$. 
\item\label{aff:sh:con:i:iii}
$\gap\in (\parl U)^\perp\cap(\parl V)^\perp.$
\item\label{aff:sh:con:i:i}
$(\forall x\in X)$ $(\forall \alpha\in \RR)$ $P_Ux=P_U(x+\alpha\gap)$.
 \item\label{aff:sh:con:i:i:V}
$(\forall x\in X)$ $(\forall \alpha\in \RR)$ 
$P_Vx=P_V(x+\alpha\gap)$.
\item\label{aff:sh:con:i:ii} 
$T_{-\gap}=\gap+T=T_{N_U,N_V(\cdot-\gap)}
=\DRS{U}{\gap+V}$.
\item\label{aff:sh:con:iii:i}
$Z_{\gap}=U\cap(\gap+V)$.
\item\label{aff:sh:con:iii} %\cref{aff:sh:con}\ref{aff:sh:con:iii}
$K_{\gap}= (\parl U)^\perp\cap(\parl V)^\perp$.
\item\label{aff:sh:con:iv}  
$\fix(T_{-\gap})
=\fix (\gap+T)=Z_{\gap}+K_{\gap}
=(U\cap(\gap+V))
+( (\parl U)^\perp\cap(\parl V)^\perp)$.
\end{enumerate}
\end{prop}
\begin{proof}
\ref{aff:sh:con:i}: Note that 
$J_A=P_U$ and $J_B=P_V$ are
affine (see e.g. \cite[Corollary~3.20(i)]{BC2011}). 
Using \cref{def:T:cones} 
we have 
$T=\Id-P_U+P_V(2P_U-\Id)=\Id-P_U+2P_VP_U-P_V$.
Since the class of affine
operators is closed under addition, 
subtraction and composition
we deduce that  $T$ is affine.

\ref{aff:sh:con:i:iii}:
It follows from \cite[Proposition~2.7~\&~Remark~2.8(ii)]{BCL04}
that 
$
\gap\in (\rec U)^\oplus\cap(\rec V)^\ominus
=(\parl U)^\perp\cap(\parl V)^\perp,
$
where the last equality follows from
\cite[Proposition~6.22 and Proposition~6.23(v)]{BC2011}.

\ref{aff:sh:con:i:i} and \ref{aff:sh:con:i:i:V}:
Combine \ref{aff:sh:con:i:iii} with 
\cref{l:nor:shf}\ref{l:nor:shf:iv}. 

\ref{aff:sh:con:i:ii}  
 It follows from \ref{aff:sh:con:i:i:V} with $\alpha$ replaced by
 $-1$, and \cref{l:nor:shf}\ref{l:nor:shf:ii}
that $J_{N_V(\cdot-\gap)}=\gap+P_V(\cdot-\gap)=\gap+P_V$.
Consequently, using  \cref{P:aff}\ref{P:aff:v}
 and \cref{def:T}
we have  
$T_{-\gap}=\gap+T =\Id-P_U+\gap+P_VR_U
=T_{N_U,N_V(\cdot-\gap)}$.
Finally,
\cref{l:nor:shf}\ref{l:nor:shf:i} 
implies that
$T_{N_U,N_V(\cdot-\gap)}=T_{N_U,N_{\gap+V}}=T_{U,\gap+V}$.

\ref{aff:sh:con:iii:i}: See \cite[Proposition~3.16]{Sicon2014}.

\ref{aff:sh:con:iii}: 
Let $z\in U\cap(\gap +V)=Z_{\gap}$
and note that, as subdifferential operators, 
$N_U$ and $N_V$ are 
paramonotone (see, e.g., \cite{Iusem98}) and 
so are the translated operators 
$-\gap+N_U$ and $N_V(\cdot-\gap)$.
 Therefore, in view of
\cite[Remark~5.4]{JAT2012}
and \ref{aff:sh:con:i:iii} we have 
\begin{align}
K_{\gap}&=(-\gap+N_U z)\cap(-N_V (z-\gap))
=(-\gap+(\parl U)^\perp)\cap(\parl V)^\perp\nonumber\\
&=(\parl U)^\perp\cap(\parl V)^\perp.
\end{align}
\ref{aff:sh:con:iv}:
Since 
$-\gap+N_U$ and $N_V(\cdot-\gap)$ are paramonotone,
it follows from \ref{aff:sh:con:i:ii}, \ref{aff:sh:con:iii} 
and \cite[Corollary~5.5]{JAT2012}
applied to the normal pair $(_{\gap}A, B_{\gap})$
that
$\fix(T_{-\gap})=\fix(\gap +T)=Z_{\gap}+K_{\gap}$. 
Now combine with \ref{aff:sh:con:iii:i} 
and \ref{aff:sh:con:iii}. 
 \end{proof}
 
We are now ready for our main result.
It illustrates that, even in the inconsistent case, 
the ``shadow sequence" $(P_UT^nx)_{n\in \NN}$ behaves 
extremely well because it converges to a normal solution
\emph{without prior knowledge} of the infimal displacement
vector.
The proof of \cref{thm:main:aff} relies on the work leading up to
this point as well as the convergence analysis of the consistent
case in \cite{JAT2014}.

\begin{thm}[\bf{Douglas-Rachford algorithm for two affine subspaces}]\label{thm:main:aff}
Let $x\in X$. Then $(\forall n\in \NN)$ we have
\begin{equation}
P_U T^n x
=P_U (T^n x+n\gap)=P_U((T_{-\gap})^nx)
=P_UT^n_{U,\gap+V}
=J_{-\gap+N_U}((T_{-\gap})^nx),
\end{equation}
and 
\begin{equation}\label{e:nearest:z}
P_U T^n x\to P_{Z_{\gap}} x=P_{U\cap (\gap+V)} x. 
\end{equation}
Moreover, if $\parl U+\parl V$ is closed 
(as is always the case when $X$ is finite-dimensional) then the 
convergence is linear 
\footnote{Recall that $x_n\to x$ \emph{linearly} with rate 
$\gamma\in \left]0,1\right[$
if $(\gamma^{-n}\norm{x_n-x})_{n\in \NN}$ is bounded.
}
 with rate being the cosine of the Friedrichs angle
\begin{equation}
c_F(\parl U,\parl V):=\sup_{\small\substack{u\in \parl U\cap W^\perp\cap{\BB{0}{1}} 
\\v\in \parl V\cap W^\perp\cap{\BB{0}{1}} }}
{\abs{\innp{u,v}}}<1,
\end{equation}
where $W=\parl U\cap \parl V$
and $\BB{0}{1}$ is the closed unit ball.
\end{thm}
\begin{proof}
Let $n\in \NN$.
Using \Cref{aff:sh:con}\ref{aff:sh:con:i:i}
with $(x,\alpha)$ replaced 
by $(T^nx,n)$ we learn that 
$P_UT^nx= P_U(T^nx+n\gap)$.
Now combine with \cref{P:aff}\ref{P:aff:iv} 
to get the second identity.
The third identity follows from applying 
\cref{aff:sh:con}\ref{aff:sh:con:i:ii}.
Finally note that using the first identity,
 \Cref{aff:sh:con}\ref{aff:sh:con:i:i} 
with $(x,\alpha)$ replaced 
by $((T_{-\gap})^nx,1)$
and \cref{l:nor:shf}\ref{l:nor:shf:iii}
we learn that $P_UT^nx= P_U((T_{-\gap})^nx+\gap)
=J_{-\gap+N_U}((T_{-\gap})^nx)
$.
Now we prove \cref{e:nearest:z}. 
It follows from \cref{gap:assmp}, 
\cref{Fact:collect}\ref{Fact:collect:non:fady:Z} 
and \cref{aff:sh:con}\ref{aff:sh:con:iii:i}
that $Z_{\gap}=U\cap(\gap+V)\neq \fady$.
Now apply \cite[Corollary~4.5]{JAT2014}. 
\end{proof}
\begin{figure}[H]%\label{fig:com:lim}
 \centering
 \includegraphics[scale=0.32]{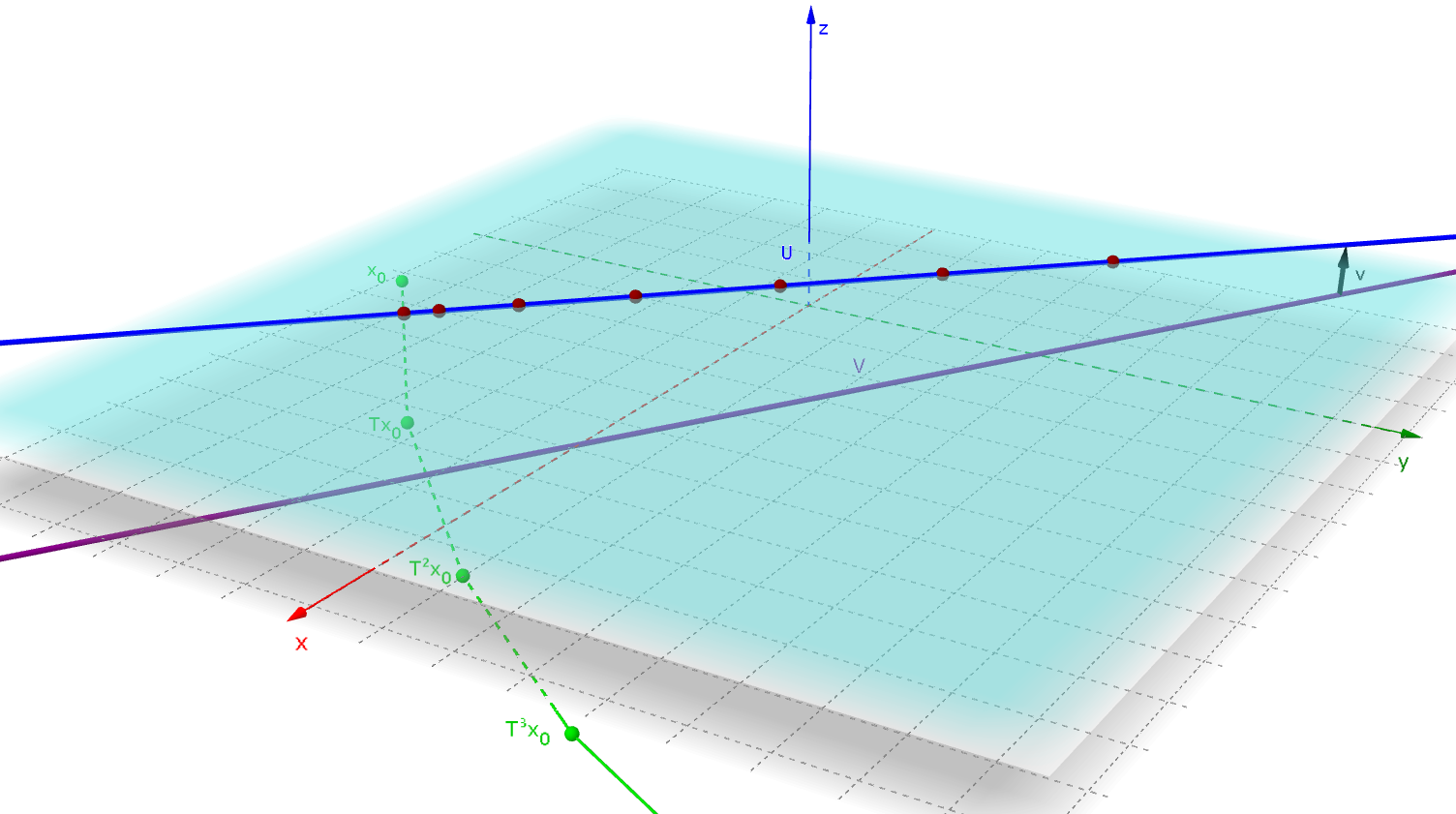}
 \caption{
 Two nonintersecting affine subspaces
 $U$ (blue line) and
 $V$ (purple line) in $\RR^3$.
 Shown are also the first few iterates of
 $(T^n x_0)_{n\in \NN}$ (green points) and 
  $(P_UT^n x_0)_{n\in \NN}$ (red points).}
\label{fig:com:shad}
\end{figure}
Figure \ref{fig:com:shad} shows 
a Geogebra snapshot 
\cite{geogebra} of
the Douglas--Rachford iterates and its shadows for 
two nonintersecting nonparallel lines $U$ and $V$ in $\RR^3$.

The following result is known
(see e.g., \cite[Corollary~1.5]{Br-Reich77} 
and \cite[Corollary~2.3]{Ba-Br-Reich78}).
We include a simple proof for completeness in \nameref{Appp:2}.

\begin{prop}\label{T:Asyp:reg}
Suppose that $T:X\to X$ be firmly nonexpansive, and
that $\gap=P_{\overline{\ran}(\Id-T)}0\in {\ran}(\Id-T)$. Then 
\begin{equation}
(\forall x\in X)\quad T^nx-T^{n+1}x\to \gap.
\end{equation}
\end{prop}

\begin{prop}[\bf{When only one set is an affine subspace}]
Suppose that $U$ is an affine subspace of $X$,
and that $T=\DRS{U}{V}$. Then for every $x\in X$
the sequence $(P_U T^n x)_{n\in \NN}$
is asymptotically regular, i.e., 
$P_U T^n x-P_U T^{n+1} x\to 0$.
\end{prop}

\begin{proof}
Using \cite[Remark~2.8(ii)]{BCL04} we have
$\gap\in (\parl U)^\perp$.
It follows from \cref{l:nor:shf}\ref{l:nor:shf:iv}
applied with $(x,\alpha)$ replaced by 
$(T^{n+1}x,1)$
and \cref{T:Asyp:reg}
 that
\begin{equation}
\norm{P_U T^n x-P_U T^{n+1} x}
=\norm{P_U T^n x-P_U (T^{n+1} x+\gap)}
\le \norm{T^n x-T^{n+1} x-\gap}\to 0,
\end{equation}
as claimed.
\end{proof}
\begin{example}{\bf (The \emph{dual} shadows)}
Consider the case when 
$U$ and $V$ are affine subspaces
of $X$ such that $U\cap V=\fady$.
Set $\widetilde{A}:=N_U^{-1}$ and $\widetilde{B}:=N_V^{-\ovee}$.
Then $\widetilde{A}^{-1}=N_U$, $\widetilde{B}^{-\ovee}=N_V$, 
 and using
\cref{Fact:collect:T:self}
we have
$T_{(\widetilde{A},\widetilde{B})}
=T_{(\widetilde{A}^{-1},\widetilde{B}^{-\ovee})}=\DRS{U}{V}$.
Moreover the inverse resolvent identity
(see, e.g., \cite[Lemma~12.14]{Rock98}) implies that
$(\forall x\in X)$ $J_{\widetilde{A}} T^n x=
(\Id-P_U)T^n x=T^n x-P_U T^n x$.
Note that $K=U\cap V=\fady$, 
hence by \cref{Fact:collect:Z:fix}
 and \cref{Fact:collect:T:self} 
$\fix T_{(\widetilde{A},\widetilde{B})}=\fady$.
Using \cite[Corollary~6(a)]{Pazy} 
we learn that for every $x\in X$ we have $\norm{T^n x}\to \infty$.
Moreover, in view of \cref{gap:assmp},
using \cite[Theorem~3.13(iii)]{BCL04}
we know that for every $x\in X$ we have 
$(P_U T^n x)_{n\in \NN}$ is a bounded sequence.
Therefore, 
    $\norm{J_{\widetilde{A}} T^n x}=\norm{T^n x-P_U T^n x}
    \ge \norm{T^n x}-\norm{P_U T^n x}\to \infty$. 

\end{example}

We conclude with the following example which 
shows that for two affine
(but not normal cone) operators
the shadows need not converge.

\begin{example}\label{ex:not:NC}
Suppose that $X=\RR^2$ and let 
$S:\RR^2\to \RR^2:(x_1,x_2)\mapsto (-x_2,x_1)$,
be the counter-clockwise rotator by $\pi/2$.
Let $b\in \RR^2\smallsetminus\stb{(0,0)}$. 
Suppose that $A:=S$ and set $B:=-S+b$. 
Then $\zer A\neq \fady$, $\zer B\neq \fady$ 
yet $\zer(A+B)=\fady$.
Moreover, $\gap=(\Id+S)(b)$, 
the set of normal solutions $Z_{\gap}=\RR^2$
and  for every $ x\in \RR^2$ we have
 $\norm{J_A T^nx}\to \infty$.
\end{example}
\begin{proof}
Let $x\in \RR^2$ and note that $S$ and $-S$ are both linear, 
continuous, single-valued,
 monotone and $S^2=(-S)^2=-\Id$.
 It follows from \cite[Proposition~2.10]{Sicon2014}
 that $J_Ax=J_Sx=\tfrac{1}{2}(\Id-S)x=
\tfrac{1}{2}(x-Sx) $. Similarly using
\cite[Proposition~23.15(ii)]{BC2011}
we can see that
$J_B x=\tfrac{1}{2}(x-b+Sx-Sb)$.
Therefore we have $ R_Ax=-Sx$
and $R_Bx=-b+Sx-Sb$. 
Hence $R_BR_Ax=S(-Sx)-Sb-b
=-S^2x-Sb-b=x-Sb-b=x-(\Id+S)b$.
Consequently
we have
\begin{equation}\label{T:skew}
(\forall x\in \RR^2) \quad Tx=\tfrac{1}{2}(\Id+R_BR_A)x=x-\tfrac{1}{2}(\Id+S)b.
\end{equation}
It follows from \cref{T:skew} that $\ran(\Id-T)=\{\tfrac{1}{2}(\Id+S)b\}$, 
hence $\gap=\tfrac{1}{2}(\Id+S)b$ and $Tx=x-\gap$. 
Therefore, using \cref{P:aff}\ref{P:aff:v} 
$\fix(\gap+T)=\fix(T_{-\gap})=\RR^2$.
Moreover, using \cref{Fact:collect:Z:fix} and 
\cref{Fact:collect:normal:T}
applied to the normal pair
$\bk{_{\gap}A,B_{\gap}} $ we learn that
 $Z_{\gap}=J_{_{\gap} A}(\fix(T_{-\gap}))=\RR^2$.
 In view of \cref{prop:fix:inc}\ref{prop:fix:inc:i}
 we have $T^n x=x-n\gap$.
 Hence, using that $J_A $ is linear, we get 
 $J_A T^n x=J_A(x-n\gap)=J_A x-nJ_A\gap$.
Now 
 $J_A\gap=\frac{1}{2}(\Id-S)(\frac{1}{2}(\Id+S)b)=\frac{1}{2}b\neq(0,0)$,
which completes the proof.
 \end{proof}

\appendix

\section*{Appendix A}\label{Appp:1}
\begin{myproof}[Proof of \cref{lem:norm:proj}.]
Using that $P_C$ is firmly nonexpansive we have
$\normsq{c-P_C0}=\normsq{P_Cc-P_C0}
\le \normsq{c-0}-\normsq{(\Id-P_C)c-(\Id-P_C)0}
= \normsq{c}-\normsq{c-P_Cc+P_C0}= \normsq{c}-\normsq{P_C0}=0$.
\end{myproof}
\section*{Appendix B}\label{app:3}
\begin{myproof}[Proof of \cref{exmp:not:fejer}.]
Clearly 
\begin{equation}
\Id-T=P_{\left[\alpha,\beta\right]}:\RR\to \RR
:x\mapsto
\begin{cases}
\alpha,&\text{~\rm if~}x\le \alpha;\\
x,&\text{~\rm if~}\alpha<x\le\beta;\\
\beta,&\text{~\rm if~}x> \beta.
\end{cases}
\end{equation}
Therefore, $\ran(\Id-T)=[\alpha,\beta]$, and consequently
$\gap=\alpha$. Moreover
\begin{equation}\label{e:T:mono:ex}
(\forall x\in \RR)\quad x\ge Tx+\alpha\ge T^2x
+2\alpha\ge\cdots\ge T^nx+n\alpha\ge \cdots .
\end{equation}
It is clear from \cref{ex:short} that 
\begin{equation}\label{e:fix:v+T}
\fix (\gap +T)=\left] -\infty,\alpha\right]. 
\end{equation}
The statement for $\fix T_{-\gap}$ then
follows from combining
 \cref{e:fix:v+T} and  \cref{lem:in:out:T}\ref{lem:in:out:T:i}.
The convergence of the sequence follows from 
\cref{ex:short} or
\cref{prop:RR:conv}\ref{prop:RR:conv:i}.
Now we prove \cref{e:all:cases}.
We claim that 
\begin{equation}\label{e:all:cases:nog}
T^n:\RR\to\RR:x\mapsto
\begin{cases}
x-n\alpha,&\text{~\rm if~}x\le \alpha;\\
(1-n)\alpha,&\text{~\rm if~}\alpha<x\le \beta;\\
x-n\beta,&\text{~\rm if~}x>\beta \text{~\rm and~} n\le \lfloor {x}/{\beta} \rfloor;\\
\min\stb{\alpha,x-\Big\lfloor \tfrac{x}{\beta} \Big\rfloor\beta}
+\bk{\Big\lfloor  \tfrac{x}{\beta} \Big\rfloor-n}\alpha,
&\text{~\rm if~}x>\beta \text{~\rm and~} n> \lfloor {x}/{\beta} \rfloor.
\end{cases}
\end{equation}
Using induction it is easy to verify the cases when 
$x\le \alpha$ and when $\alpha<x\le \beta$.
Now we focus on the case when $x>\beta$.
Set 
\begin{equation}
K:=\lfloor x/\beta\rfloor \text{~~and~~} 
r:=x-K\beta,
\end{equation}
and note that $x=K\beta+r$, $K\in \{1,2,3,\ldots\}$ 
and $0\le r<\beta$. In view of \cref{e:T:mono:ex}, 
if $n\in\{0,1,2,3,\ldots,K\}$ we 
get $T^n x=x-n\beta =(K-n)\beta+r$.
In particular,
\begin{equation}\label{e:K:r}
T^K x=x-\lfloor x/\beta\rfloor \beta=r.
\end{equation}
If $n>K$  we examine two cases.
Case 1: $0\le r\le \alpha$. It follows from \cref{e:K:r} and
 \cref{ex:not:fejer} that $(\forall n\ge K)$ $T^n x=r+(K-n)\alpha$.
Case 2: $ \alpha< r<\beta$.
Note that $T^{K+1} x=0$, therefore using \cref{e:K:r} and
\cref{ex:not:fejer} we have 
 $(\forall n>K)$ $T^n x=(K+1-n)\alpha=\alpha+(K-n)\alpha$,
 which proves \cref{e:all:cases:nog}.
 Now \cref{e:all:cases} follows from \cref{e:all:cases:nog}
 because $\gap=\alpha$.
Letting $n\to \infty$ in \cref{e:all:cases} yields 
\cref{e:all:cases:lim}.
Note that $\min\{\alpha,x-\lfloor  \tfrac{x}{\beta}\rfloor\beta \}\ge0$
and $\lfloor  \tfrac{x}{\beta}\rfloor\ge 1$.
By considering cases $(K=1$ and $K\ge 1)$,
 \cref{e:all:cases:lim} implies that
$\lim_{n\to \infty}(T^n x_0+n\gap)=
\min\{\alpha,x-\lfloor  \tfrac{x}{\beta}\rfloor\beta \}
+\lfloor  \tfrac{x}{\beta}\rfloor\alpha >\alpha
\not\in \left]-\infty,\alpha\right]=\fix(\gap+T)$.

\end{myproof}

\section*{Appendix C}\label{app:D}

\begin{myproof}[Proof of \cref{e:asym}.]
Considering cases, we easily check that
\begin{equation}\label{ex:ran:disp}
\Id-T\colon\RR\to\RR: x\mapsto
\bk{1-\alpha}\max\stb{x,\beta}+\alpha\beta \ge \beta>0.
\end{equation}
Hence $\fix T=\fady$ and $\gap=\beta $ as claimed.
Moreover, using \cref{ex:ran:disp} one can verify that
 \begin{equation}\label{e:abl:beta}
 (\forall x\in X)\quad x\ge Tx+\beta>\cdots
 \ge T^nx+n\beta\ge T^{n+1}x+(n+1)\beta\ge \cdots  .
 \end{equation}
We also verify that
\begin{equation}\label{eq:defn:in}
(\forall x\in \RR)\quad T_{-\gap} :\RR\to\RR: x\mapsto
\max\stb{x,0}\alpha+\min\stb{x,0}.
\end{equation}
We now prove \cref{ex:inn:sft} by induction.
Let $x\in \RR$.

Clearly when $n=0$ the base case holds true.
Now suppose that for some $n\in \NN$
\cref{ex:inn:sft} holds.
If $ x\le 0$ then $\bk{T_{-\gap}}^n x=x\le 0$, 
and therefore, \cref{eq:defn:in} implies that
$\bk{T_{-\gap}}^{n+1} x
=T_{-\gap}(\bk{T_{-\gap}}^n x)=T_{-\gap}x=x$.
Similarly we have
$ x> 0 \RA \alpha^nx=\bk{T_{-\gap}}^n x> 0$, 
and consequently
\cref{eq:defn:in} implies that
 $\bk{T_{-\gap}}^{n+1} x=T_{-\gap}(\bk{T_{-\gap}}^n x)
 =T_{-\gap}(\alpha^n x)=\alpha^{n+1} x$.
The proof of \cref{ex:out:sft} follows 
from combining \cref{ex:inn:sft}
and \cref{lem:in:out:T}\ref{lem:in:out:T:ii}.
Now we turn to \cref{ex:nv:sft}.
We consider two cases.\\
 {Case 1: $x\le \beta$}. 
It is obvious using the definition of $\T $ that
  $(\forall n\in \NN)$  $\T^n x=x-n\beta$.\\
 {Case 2: $x> \beta$}. 
 Let $n\in \NN$ be such that $T^{n}x>\beta$.
By \cref{e:abl:beta} 
and \cref{eq:ex:long} we have
\begin{align}\label{e:geom:ser}
  T^{n+1} x
  &= \alpha^{n+1} x -(\alpha^{n+1}+\alpha^{n}+\cdots+\alpha)
 \beta=\alpha^{n+1}x -\frac{\alpha(1-\alpha^{n+1})}{1-\alpha}\beta\nonumber\\
 &=\alpha^{n+1}\bk{\frac{(1-\alpha)x+\alpha\beta}{1-\alpha}}
 -\frac{\alpha}{1-\alpha}\beta.
 \end{align}
In view of \cref{e:abl:beta} there exists a unique integer, 
say, $ q(x)\in \{1,2,\ldots\}$ 
 that satisfies
 $\T^{q(x)-1}x>  \beta$
 and $\T^{q(x)} x\le\beta$. 
Since $0<\alpha<1$, 
using \cref{e:geom:ser}
we have
 \begin{align*}
 T^{q(x)}x\le \beta
 &\iff \alpha^{q(x)}\bk{\frac{(1-\alpha)x+\alpha\beta}{1-\alpha}}
 -\frac{\alpha}{1-\alpha}\beta\le \beta\\
&\iff \alpha^{q(x)}\bk{\frac{(1-\alpha)x+\alpha\beta}{1-\alpha}}
\le \frac{\beta}{1-\alpha}
\iff  \alpha^{q(x)}\bk{(1-\alpha)x+\alpha\beta}\le \beta\\
&\iff \alpha^{q(x)}\le \frac{\beta}{(1-\alpha)x+\alpha\beta}
\iff {q(x)\ge{\log_\alpha \frac{\beta}
{\alpha\beta+(1-\alpha)x}}}
.
 \end{align*}
 Consequently, 
$q(x)=\left\lceil {\log_\alpha \frac{\beta}
{\alpha\beta+(1-\alpha)x}}\right \rceil $.
 At this point, since $\T^{q(x)} x \le \beta$, 
 we must have $(\forall n\ge q(x))$  $\T^n x =\T^{q(x)} x-(n-q(x))\beta$,
 which proves \cref{ex:nv:sft}.
The formulae \cref{ex:inn:sft:lim}, 
\cref{ex:out:sft:lim} and \cref{ex:nv:sft:lim}
are direct consequences of
 \cref{ex:inn:sft}, \cref{ex:out:sft} and \cref{ex:nv:sft}, 
 respectively.
To prove the last claim note that 
if 
$
S:\RR\to \RR
$
is such that for every $n\in \NN$
we have
$
S^n=T^n+n\gap
$,
then setting $n=1$ must yield
 \begin{equation}\label{eq:ex:last}
S=\gap+T\colon\RR\to\RR: x\mapsto
\begin{cases}
x,&x\le\beta;\\
\alpha\bk{x-\beta}+\beta,& x>\beta.
\end{cases}
\end{equation}
Now compare 
\cref{ex:out:sft} and \cref{ex:nv:sft}.
\end{myproof}
\section*{Appendix D}\label{Appp:2}
\begin{myproof}[Proof of \cref{T:Asyp:reg}.]
Let $y_0\in \fix(\gap+T)$ and note that
\cref{prop:fix:inc}\ref{prop:fix:inc:i:i} implies 
that $(\forall n\in \NN)~(\Id-T)T^n y_0=\gap$. 
Since $T $ is firmly nonexpansive,
it follows from 
\cref{prop:fix:inc}\ref{prop:fix:inc:ii:ii}
and
\cite[Proposition~5.4(ii)]{BC2011}
that
\begin{align}
\normsq{T^nx-T^{n+1}x- \gap}
&=\normsq{(\Id-T)T^n x-(\Id-T)T^ny_0}\nonumber\\
&\le \normsq{T^n x-T^n y_0}- \normsq{T^{n+1} x-T^{n+1} y_0}\nonumber\\
&=\normsq{T^{n}x+n\gap-y_0}-\normsq{T^{n+1}x+(n+1)\gap-y_0}\to 0.\nonumber
\end{align}
\end{myproof}

\end{document}